\documentclass[leqno,english]{amsart}
\usepackage{amsfonts,amssymb,amsmath,amsgen,amsthm}

\usepackage{tikz,tkz-tab}
\usepackage{hyperref}
\usepackage{color}
\newcommand{\msc}[2][2000]{%
  \let\@oldtitle\@title%
  \gdef\@title{\@oldtitle\footnotetext{#1 \emph{Mathematics subject
        classification.} #2}}%
}

\theoremstyle{plain}
\newtheorem{theorem}{Theorem}[section]

\newtheorem{lemma}[theorem]{Lemma}
\newtheorem{corollary}[theorem]{Corollary}
\newtheorem{proposition}[theorem]{Proposition}

\theoremstyle{remark}
\newtheorem{remark}[theorem]{Remark}

\newtheorem{example}[theorem]{Example}

\def\C{{\mathbb C}}
\def\R{{\mathbb R}}
\def\N{{\mathbb N}}
\def\Z{{\mathbb Z}}
\def\T{{\mathbb T}}
\def\H{{\mathcal H}}
\def\O{\mathcal O}
\def\F{\mathcal F}

\def\({\left(}
\def\){\right)}
\def\<{\left\langle}
\def\>{\right\rangle}
\def\le{\leqslant}
\def\ge{\geqslant}

\def\Tend#1#2{\mathop{\longrightarrow}\limits_{#1\rightarrow#2}}

\def\d{{\partial}}
\def\eps{\varepsilon}

\def\si{{\sigma}}

\def\1{{\mathbf 1}}
\def\w{{w}}

\DeclareMathOperator{\RE}{Re}
\DeclareMathOperator{\IM}{Im}

\def\E{{\mathbb E}}

\def\FF{F}
\def\GG{G}

\numberwithin{equation}{section}

\newcommand{\mynegspace}{\hspace{-0.12em}}
\newcommand{\vvvert}{\rvert\mynegspace\rvert\mynegspace\rvert}

\begin{document}

\title[WKB analysis of non-elliptic NLS]{WKB analysis of
  non-elliptic nonlinear Schr\"odinger equations}   

\author[R. Carles]{R\'emi Carles}
\address[R. Carles]{Univ Rennes, CNRS\\ IRMAR, UMR 6625\\ F-35000 Rennes\\ France}
\email{Remi.Carles@math.cnrs.fr}
\author[C. Gallo]{Cl\'ement Gallo}
\address[C. Gallo]{IMAG\\ Univ Montpellier, CNRS\\ Montpellier\\ France}
\email{Clement.Gallo@umontpellier.fr}

\begin{abstract}
  We justify the WKB analysis for generalized nonlinear Schr\"odinger
  equations (NLS), including the hyperbolic NLS and the
  Davey-Stewartson II system. Since the leading order system in this
  analysis is not hyperbolic, we work with analytic regularity, with a
  radius of analyticity decaying with time, in order to obtain better
  energy estimates. This provides qualitative information regarding
  equations for which global well-posedness in Sobolev spaces is
  widely open. 
\end{abstract}
\maketitle

\section{Introduction}

\subsection{Motivation}
\label{sec:motivation}
The two-dimensional ``hyperbolic'' nonlinear Schr\"odinger equation,
\begin{equation}
  \label{eq:NLShyp}
  i\d_t \psi + \frac{1}{2}\d_{1}^2 \psi - \frac{1}{2}\d_{2}^2 \psi \pm |\psi|^2\psi=0, \quad
  (x_1,x_2)\in \R^2,
\end{equation}
appears in nonlinear optics (see e.g. \cite{DLS16,Sulem}), but remains quite
mysterious as far as analysis is concerned: it is locally well-posed
in $H^s(\R^2)$ for any $s>0$, it is $L^2$-critical,
hence locally well-posed in $L^2(\R^2)$ (with a suitable definition of
local well-posedness in this critical case), but apart from the small
data case, the global existence issue remains a delicate issue in such
spaces, even though refined Strichartz estimates are available
\cite{RoVa06}, because the conserved energy is not a positive
functional,
\begin{equation*}
  E=\|\d_1 \psi\|_{L^2(\R^2)}^2 - \|\d_2 \psi\|_{L^2(\R^2)}^2 \mp
  \|\psi\|_{L^4(\R^2)}^4. 
\end{equation*}
Note that the sign  
of the nonlinearity is rather irrelevant, since we may exchange the
roles of $x_1$ and $x_2$. However, global existence in $H^s(\R^2)$ for $s>0$ is
obtained through modulation approximation in \cite{Totz16}. 
On the other hand, global solutions under
the form of spatial standing waves have been
studied in \cite{CoFi-p,KNZ11}, along with their stability.
\smallbreak

Similarly, the Davey--Stewartson system
\begin{equation}
  \label{eq:DS2}
  \left\{
    \begin{aligned}
    & i\d_t \psi + \frac{1}{2}\d_{1}^2 \psi - \frac{1}{2}\d_{2}^2 \psi
    =  \(\chi |\psi|^2 
     +\omega\d_1 \phi\)\psi, \quad
     (x_1,x_2)\in \R^2,\quad \chi,\omega\in \R,\\
     &\d_1^2\phi+\d_2^2\phi = \d_1 |\psi|^2,
     \end{aligned}
  \right.
\end{equation}
is locally well-posed in the same spaces, $L^2$-critical, and enjoys a
Hamiltonian structure with an energy whose sign is indefinite.
Indeed, \eqref{eq:DS2} can be rewritten
\begin{equation*}
  i\d_t \psi + \frac{1}{2}\d_{1}^2 \psi - \frac{1}{2}\d_{2}^2 \psi
    =  \(\chi |\psi|^2 
     +\omega  K\ast |\psi|^2\)\psi,
   \end{equation*}
   where the symmetric kernel is such that
   \begin{equation*}
     \widehat K(\xi) = \frac{\xi_1^2}{\xi_1^2+\xi_2^2}. 
   \end{equation*}
On the
other hand, for a suitable combination of the coefficients $\chi$ and
$\omega$, that is $2\chi+\omega=0$, \eqref{eq:DS2} is completely
integrable (see
e.g. \cite{GhSa90,KlSa-p}). Global well-posedness and scattering in
$L^2$ for the defocusing case were recently established in
this specific case thanks to inverse scattering and harmonic analysis
techniques, see \cite{NRT-p}. 
In this note, we justify the approximation
of such equations in a high frequency regime, known as semi-classical
limit, this giving some extra information concerning the dynamics
associated to these equations.

\subsection{Setting}
\label{sec:setting}

We consider the equation, including both \eqref{eq:NLShyp} and \eqref{eq:DS2},
\begin{equation}
  \label{eq:nls}
  i\eps\d_t u^\eps +\frac{\eps^2}{2}D^2 u^\eps
  +i\eps\<\beta,\nabla\left[g\(|u^\eps|^{2}\)u^\eps\right]\>=Vu^\eps+
  \sum_{j=1}^J \left(K_j*|u^\eps|^{2\sigma_j}\right)u^\eps,
\end{equation}
in the semi-classical limit $\eps \to 0$, where $\eps>0$, $T>0$, $d\ge
1$ is the spatial dimension, $J\ge 1$, and $(t,x)\in[0,T]\times \E^d $. More specifically,
\begin{itemize}
\item $u^\eps(t,x)\in\C$ is the wave function 
\item $\mathbb{E}^d$ can be either $\R^d$ or to the torus $\T^d=\(\R/2\pi\Z\)^d$, 
\item \hfill
  $D^2=\underset{j,k=1}{\overset{d}{\sum}}\eta_{j,k}\partial_j\d_k=\<\nabla,
  H\nabla\>,$\hfill\null\\ where $H=(\eta_{j,k})_{1\le j,k\le d}$ is a
  symmetric (not necessarily positive or invertible) real matrix, and
  $\<\cdot,\cdot\>$ denotes the inner product on $\R^d$,
\item $\beta=(\beta_j)_{1\le j\le d}\in\R^d$ and $g(s)=\alpha s^\gamma$, where $\alpha\in\R$ and $\gamma\in\N\string\ \{0\}$. We consider such a function $g$ in order to simplify the notations, but our method also works if $g$ is not a monomial.
\item For $j\in \{1,\dots,J\}$, $\sigma_j\in\N\string\ \{0\}$ is an
  integer, and $K_j$ denotes a tempered distribution with a
   bounded Fourier transform $\widehat{K_j}\in L^\infty(\E^d)$. This
  covers the case where $K=\delta$,  typically as in \eqref{eq:NLShyp}. 
\item $V=V(t,x)$ is a potential. $V$ is supposed to be analytic in the
  $x$ variable. More precisely, we assume that $V$ belongs to the
  space $L^2_{T_0}\H_{w_0}^{\ell+1/2}$ for some $T_0>0$, $w_0>0$,
  $\ell>(d+1)/2$, a space that will be defined below. 
\end{itemize}
Our motivation for considering the case $\E^d=\T^d$ lies in the fact
that numerical simulations are often performed in a periodic box:
unless suitable absorbing boundary conditions are imposed, the
observed dynamics is that of \eqref{eq:nls} on $\T^d$, which is fairly
different from the one on $\R^d$. 
\begin{remark}
  In view of the
  assumption that is usually made on $V$ in order to get $H^s$
  solutions (namely, $\d^\alpha V\in L^\infty$ for $2\le|\alpha|\le
  s$, see e.g. \cite{CaBKW}), it is reasonable to ask for analyticity of
  $V$ in order to get analytic solutions.  
\end{remark}
\begin{remark}
  In a similar fashion as we consider an external potential, our
  analysis exports to the magnetic case, where
  \[D^2=
    \underset{j,k=1}{\overset{d}{\sum}}\eta_{j,k}\(\partial_j-iA_j\)\(\d_k-iA_k\),\]
  provided that the magnetic potentials $A_j$ are analytic (in the same
  sense as for $V$). 
\end{remark}

The initial data that we 
consider are WKB states:
\begin{equation}
  \label{eq:ci-u}
  u^\eps(0,x)=a_0^\eps(x)e^{i\phi_0^\eps(x)/\eps}=:u_0^\eps(x),
\end{equation}
where $\phi_0^\eps:\E^d\to \R$ is a real-valued phase, and $a_0^\eps:\E^d\to \C$
is a possibly complex-valued amplitude. We emphasize that our approach
is distinct from the polar decomposition known as Madelung transform,
hence the possibility for the amplitude to be (or become) complex.
Our goal is to understand the 
semi-classical limit of equation (\ref{eq:nls}), that is to describe
the behavior in the limit $\eps\to 0$ of the solutions to
(\ref{eq:nls}) with initial data (\ref{eq:ci-u}). Generalizing the
idea of \cite{Grenier98}, we remark that if
$(\phi^\eps,a^\eps)$ solves the system 
\begin{equation}\label{eq:syst-grenier}
  \left\{
    \begin{aligned}
      &\d_t \phi^\eps + \frac{1}{2}\<\nabla
      \phi^\eps, H\nabla\phi^\eps\>+g(|a^\eps|^2)\<\beta,
      \nabla\phi^\eps\>+\sum_{j=1}^J K_j \ast
      |a^\eps|^{2\sigma_j}+V=0  ,\\
    &  \qquad \qquad \phi^\eps_{\mid t=0}=\phi_0^\eps,\\ 
& \d_t a^\eps +\<\nabla \phi^\eps , H\nabla a^\eps\>
+\frac{1}{2}a^\eps D^2 \phi^\eps+\<\beta,
\nabla\left[g(|a^\eps|^2)a^\eps\right] \>=\frac{i\eps}{2}D^2 a^\eps,\\
&\qquad \qquad a^\eps_{\mid t=0}=a_0^\eps,
    \end{aligned}
\right.
\end{equation}
then 
\[u^\eps(t,x)=a^\eps(t,x)e^{i\phi^\eps(t,x)/\eps}\]
solves (\ref{eq:nls})-(\ref{eq:ci-u}). Therefore,  we focus on
(\ref{eq:syst-grenier}).  Note that $\phi^\eps$, there, remains
real-valued, while $a^\eps$ will be complex-valued
(even if $a_0^\eps$ is real), due to the term $i\eps D^2 a^\eps$,
which is a remain of dispersive effects in the initial Schr\"odinger
equation.
\begin{example}
  In the case of the hyperbolic NLS \eqref{eq:NLShyp}, $d=2$ and the equations
  in \eqref{eq:syst-grenier} read
  \begin{equation*}
    \left\{
      \begin{aligned}
        &\d_t \phi^\eps +
        \frac{1}{2}\(|\d_1\phi^\eps|^2-|\d_2\phi^\eps|^2\)\mp |a^\eps|^2=0  ,\\
& \d_t a^\eps +\d_1 \phi^\eps \d_1 a^\eps-\d_2 \phi^\eps \d_2 a^\eps
+\frac{1}{2}a^\eps \(\d_1^2
\phi^\eps-\d_2^2\phi^\eps\)=\frac{i\eps}{2}\(\d_1^2a^\eps-\d_2^2a^\eps\). 
      \end{aligned}
    \right.
  \end{equation*}
  Passing formally to the limit $\eps\to 0$, and setting $\rho=|a|^2$, $v= \nabla
  \phi$, and $\tilde v = (\d_1\phi,-\d_2\phi)^T$, we find, respectively,
   \begin{equation*}
    \left\{
      \begin{aligned}
        &\d_t v_j + v_1\d_1v_j-v_2\d_2 v_j \mp \d_j \rho=0  ,\\
& \d_t \rho +\d_1(\rho v_1)-\d_2(\rho v_2) =0,
      \end{aligned}
    \right.
    \qquad
    \left\{
      \begin{aligned}
        &\d_t\tilde v +\<\tilde v , \nabla\> \tilde v \mp
        \begin{pmatrix}
          \d_1\rho\\
          -\d_2 \rho
        \end{pmatrix}
        =0  ,\\
& \d_t \rho +\<\nabla, \rho \tilde v\>=0.
      \end{aligned}
    \right.
  \end{equation*}
  No special structure such as symmetry (ensuring the hyperbolicity of
  the system) seems to be available here. Because of this, we work with
  analytic regularity, since Sobolev regularity is hopeless in such a
  case (see \cite{LNT18,GuyCauchy}). 
\end{example}

\smallbreak

In \cite{CaGa17b}, we have already addressed the issue of the
semi-classical limit of (\ref{eq:nls}) in the case where $d=1$,
$\E^d=\R$, $D^2=\d^2_x$, where the nonlinearity is local (that is,
$K=\delta$) and where $V=0$. We show that the
method that was used in \cite{CaGa17b} can be generalized  
\begin{itemize}
\item To any dimension of the space variable,
\item To the torus $\T^d$,
\item To a second order operator $D^2$ which is not necessarily elliptic,
\item To nonlocal nonlinearities,
\item To non-zero potentials $V$.
\end{itemize}
Note that at least for the first three aspects evoked above, our
approach provides results which can be established by following the
same strategy as in \cite{PGX93} (case $x\in \T^d$) and
\cite{ThomannAnalytic} (case $x\in \R^d$), based on the notion
of analytic symbols, as developed by J.~Sj\"ostrand \cite{Sjo82}. On
the other hand, incorporating nonlocal nonlinearities seems to be
easier when relying on a notion of (time dependent) analyticity based on Fourier
analysis, as in \cite{GV01} (and \cite{MoVi11} for the same idea
in a different context); see the next
subsection for more details.


\subsection{The functional framework.}
\label{sec:rigor1}
For $\w \ge  0$ and $\ell\ge 0$, we consider the space
\begin{equation*}
\H_\w^\ell=\{\psi\in L^2(\E^d),\quad \|\psi\|_{\H_\w^\ell}<\infty\},
\end{equation*}
where
\begin{equation*}
  \| \psi\|_{\H_\w^\ell}^2 := \left\{
    \begin{aligned}
               &\int_{\R^d}\<\xi\>^{2\ell}e^{2\w\<\xi\>}
               |\widehat\psi(\xi)|^2d\xi&\text{ if }\quad
               \E^d=\R^d,\\ 
               & \underset{m\in\Z^d}{\sum}\<m\>^{2\ell}e^{2\w\<m\>}
               |\widehat\psi(m)|^2&\text{ if }\quad \E^d=\T^d,
    \end{aligned}
  \right.
\end{equation*}
with $\<\xi\>=\sqrt{1+|\xi|^2}$, and where the Fourier transform and
series are defined by
\begin{equation*}
  \widehat \psi(\xi)=\F \psi(\xi)=\frac{1}{(2\pi)^{d/2}}\int_{\R^d} \psi(y)e^{-i \<\xi, y\>}dy\quad {\rm if}\quad \E^d=\R^d,
\end{equation*}
\begin{equation*}
\widehat{\psi}(m)=\F \psi(m)=\frac{1}{(2\pi)^{d/2}}\int_{\T^d}\psi(y)e^{-i\<m, y\>}dy\quad {\rm if}\quad \E^d=\T^d.
\end{equation*}
We obviously have the monotonicity property,
\begin{equation}\label{eq:monotone}
  0\le \w_1\le \w_2\quad \Longrightarrow\quad
  \|\psi\|_{\H_{\w_1}^\ell}\le \|\psi\|_{\H_{\w_2}^\ell}.
\end{equation}
The interest of considering a time-dependent, decreasing, weight $\w$
is that energy estimates become similar to parabolic estimates, since
\begin{equation}\label{eq:evol-norm}
  \frac{d}{dt}\|\psi\|_{\H_\w^\ell}^2 = 2\RE\(\psi,\d_t
  \psi\)_{\H_\w^\ell} +2\dot \w \|\psi\|_{\H_\w^{\ell+1/2}}^2,
\end{equation}
where $\(\cdot,\cdot\)_{\H_\w^\ell}$ denotes the natural inner product
stemming from the above definition.
We choose a weight
$\w=\w(t)=\w_0-Mt$, where $\w_0>0$ and $M>0$ are fixed. For $T>0$, we work in spaces such as
\begin{equation*}
C([0,T],\H_\w^\ell)=\left\{\psi\ |\
  \mathcal{F}^{-1}\(e^{\w(t)\<\xi\>}\widehat{\psi}\)\in
  C([0,T],\H_0^\ell)=C([0,T],H^\ell)\right\}, 
\end{equation*}
where $H^\ell=H^\ell(\E^d)$ is the standard Sobolev space, or
\begin{equation*}
L^2([0,T],\H_\w^\ell)=L^2_T\H_\w^\ell=\left\{\psi\ |\
  \int_0^T\|\psi(t)\|_{\H_{\w(t)}^\ell}^2dt<\infty\right\}.
\end{equation*}
Phases and amplitudes belong to spaces
\begin{align*}
Y_{\w,T}^\ell=C([0,T],\H_\w^{\ell})\cap
L^2_T \H_\w^{\ell+1/2},
\end{align*}
and the fact that phase and amplitude do not have exactly the same
regularity shows up in the introduction of the space
\begin{align*}
X_{\w,T}^\ell=Y_{\w,T}^{\ell+1}\times Y_{\w,T}^\ell,
\end{align*}
which is reminiscent of the fact that in the case where the operator
on the left hand side of \eqref{eq:syst-grenier} is hyperbolic
(typically, starting from a defocusing cubic Schr\"odinger equation  with the
standard Laplacian), the good
unknown is $(\nabla \phi^\eps,a^\eps)$ rather than $(\phi^\eps,a^\eps)$
(see \cite{Grenier98}). 
The space $X_{\w,T}^\ell$ is endowed with the norm 
$$\|(\phi,a)\|_{X_{\w,T}^\ell}=\vvvert\phi\vvvert_{\ell+1,T}+\vvvert
a\vvvert_{\ell,T}, $$ 
where
\begin{equation}\label{eq:triple}
  \vvvert \psi\vvvert_{\ell,t}^2 = \max\left(\sup_{0\le s\le
    t}\|\psi(s)\|_{\H_{\w(s)}^\ell} ^2,2M\int_0^t 
  \|\psi(s)\|_{\H_{\w(s)}^{\ell+1/2}}^2 ds\right). 
\end{equation}

\subsection{Main results}
\label{sec:main-results}

Our first result states local well-posedness for
(\ref{eq:syst-grenier}) in this functional framework.
\begin{theorem}\label{th:wp}
Let $\w_0>0$, $\ell>(d+1)/2$, $T_0>0$, $V\in L^2_{T_0}\H_{w_0}^{\ell+1/2}$ and $(\phi_0^\eps,a_0^\eps)_{\eps\in[0,1]}$ be
a bounded family in $\H_{\w_0}^{\ell+1}\times\H_{\w_0}^{\ell}$. Then,
provided $M=M(\ell)>0$ is chosen sufficiently large, for all $\eps\in [0,1]$,
there is a unique solution $(\phi^\eps,a^\eps)\in X_{\w,T}^\ell$ to
\eqref{eq:syst-grenier}, where $\w(t)=\w_0-Mt$ and
$T=T(\ell)<\min\(\w_0/M,T_0\)$. Moreover, up to the choice of a possibly larger value for
$M$ (and consequently a smaller one for $T$), we have the estimates
$$\vvvert\phi^\eps\vvvert_{\ell+1,T}^2\le
4\|\phi_0^\eps\|_{\H_{\w_0}^{\ell+1}}^2+\max_{1\le j\le J}\|a_0^\eps\|_{\H_{\w_0}^{\ell}}^{4\sigma_j} +\|V\|^2_{L^2_{T_0}\H_{w_0}^{\ell+1/2}},\qquad
\vvvert a^\eps\vvvert_{\ell,T}^2\le
2\|a_0^\eps\|_{\H_{\w_0}^{\ell}}^2.$$ 
\end{theorem}
An important aspect in the above statement is the fact that the local
existence time $T$ is uniform in $\eps\in [0,1]$. In view of the
discussion in Subsection~\ref{sec:setting}, this yields a uniform time
of existence for the solution of \eqref{eq:nls}. We emphasize that
this property is not a consequence of the standard local
well-posedness argument (based on a fixed point), which would yield a
local existence time $T^\eps=\O(\eps^\alpha)$ for some $\alpha\ge 1$,
while we recall that the a priori estimates do not make it possible to
extend the local solution to much larger time. In other words, the
formulation \eqref{eq:syst-grenier} is already helpful at the level
of the  life-span of the solution to \eqref{eq:nls}. 
\smallbreak

Our second result states the convergence of the phase and of the complex amplitude as $\eps\to 0$. 
\begin{theorem}\label{th:conv-obs}
Let $\w_0>0$, $\ell>(d+1)/2$, $T_0>0$, $V\in L^2_{T_0}\H_{w_0}^{\ell+3/2}$, 
$(\phi_0,a_0)\in\H_{\w_0}^{\ell+2}\times\H_{\w_0}^{\ell+1} $ and
$(\phi_0^\eps,a_0^\eps)_{\eps\in(0,1]}$ bounded in
$\H_{\w_0}^{\ell+1}\times\H_{\w_0}^{\ell}$ such that 
\begin{equation*}
  r_0^\eps:= \|\phi_0^\eps-
\phi_0\|_{\H_{\w_0}^{\ell+1}} + \|a_0^\eps-
a_0\|_{\H_{\w_0}^{\ell}} \Tend \eps 0 0.
\end{equation*}
Let $M=M(\ell+1)$ and $T=T(\ell+1)$, as defined as in Theorem~\ref{th:wp}. 
 Then there is an
 $\eps$-independent 
 $C>0$ such that for all $\eps\in (0,1]$, 
$$\vvvert\phi^\eps-\phi\vvvert_{\ell+1,T}+\vvvert
a^\eps-a\vvvert_{\ell,T}\le C\(r_0^\eps+\eps\),$$ 
where $(\phi^\eps,a^\eps)$ denotes the solution to
\eqref{eq:syst-grenier}  and
$(\phi,a)$ is the solution to the formal limit of \eqref{eq:syst-grenier} as $\eps\to 0$ 
\begin{equation}\label{eq:syst-limit}
  \left\{
    \begin{aligned}
      &\d_t \phi + \frac{1}{2}\<\nabla\phi, H\nabla
      \phi\>+g\(|a|^2\)\<\beta,\nabla\phi\>+\sum_{j=1}^J K_j
      \ast |a|^{2\sigma_j}+V=0  ,\\
      &\qquad\qquad \phi_{\mid t=0}=\phi_0,\\ 
& \d_t a + \<\nabla \phi , H\nabla a\>
+\frac{1}{2}a D^2\phi+\<\beta,\nabla\(g\(|a|^2\)a\)\>=
0,\\
&\qquad\qquad a_{\mid t=0}=a_0,
    \end{aligned}
\right.
\end{equation}
whose existence and uniqueness stem from Theorem~\ref{th:wp}.
\end{theorem}

However,  regarding convergence of the wave function $u^\eps$, the previous
result is not sufficient. Indeed, as fast as the initial data $\phi_0^\eps$ and
$a_0^\eps$ may converge as $\eps\to 0$, Theorem \ref{th:conv-obs} at most
guarantees that $\phi^\eps-\phi=\mathcal{O}(\eps)$, which only ensures
that $a^\eps e^{i\phi^\eps/\eps}-ae^{i\phi/\eps}=\mathcal{O}(1)$, due
to the rapid oscillations. However, the above convergence result
suffices to infer the convergence of quadratic observables:
\begin{corollary}\label{cor:quad}
  Under the assumptions of Theorem~\ref{th:conv-obs}, the position and
  momentum densities converge:
  \begin{equation*}
    |u^\eps|^2\Tend \eps 0 |a|^2,\quad \text{and}\quad\IM \(\eps \bar
    u^\eps \d 
    u^\eps\) \Tend \eps 0 |a|^2\d \phi,\quad\text{in
    }L^\infty([0,T];L^1\cap L^\infty(\E^d)),
  \end{equation*}
where $\d=\d_j$, for any $j\in\{1,\cdots,d\}$.
\end{corollary}

 In
order to get a good approximation of the wave function $a^\eps
e^{i\phi^\eps/\eps}$, we have to approximate $\phi^\eps$ up to an
 error which is small compared to $\eps$. It will be done by adding a
 corrective 
term to $(\phi,a)$. For this purpose, we consider the system
obtained by linearizing \eqref{eq:syst-grenier} about
$(\phi,a)$, solution to \eqref{eq:syst-limit},
\begin{equation}\label{eq:syst-grenier-linearized}
  \left\{
    \begin{aligned}
      &\d_t \phi_1 + \<\nabla \phi, H\nabla\phi_1\>+g\(|a|^2\)\<\beta,\nabla \phi_1\>+ 
        2g'\(|a|^2\)\<\beta,\nabla\phi\>\RE\left(\overline{a}a_1\right)\\
&\qquad
+2\sum_{j=1}^J\sigma _j K_j \ast
\(|a|^{2(\sigma_j-1)}\RE\left(\overline{a}a_1\right)\right)=0 
,\\
&\qquad \qquad
      {\phi_1}_{\mid t=0}=\phi_{10},\\ 
& \d_t a_1  + \<\nabla \phi, H\nabla a_1\>+\frac{1}{2}a_1 D^2 \phi +
\<\nabla a, H\nabla \phi_1  \>
+\frac{1}{2}a D^2 \phi_1  \\
&\qquad+\<\beta,\nabla\( g\(|a|^2\)a_1\)\>+2\<\beta,\nabla\(a g'\(|a|^2\)\RE\left(\overline{a}a_1\right)\)\>
 =
 \frac{i}{2} D^2 a,\\
 &\qquad \qquad{a_1}_{\mid t=0}=a_{10}.
    \end{aligned}\right.
    \end{equation}

Provided $(\phi_0,a_0)\in \H_{\w_0}^{\ell+3}\times\H_{\w_0}^{\ell+2}$
(which implies $(\phi,a)\in X_{\w,T}^{\ell+2}$ according to
Theorem \ref{th:wp}) and $(\phi_{10},a_{10})\in
\H_{\w_0}^{\ell+2}\times\H_{\w_0}^{\ell+1}$, we will see that the solution to
(\ref{eq:syst-grenier-linearized}) belongs to $X_{\w,T}^{\ell+1}$, and
our final result is the following. 
 
\begin{theorem}
\label{th:conv-wf}
Let $\w_0>0$, $\ell>(d+1)/2$, $T_0>0$, $V\in L^2_{T_0}\H_{w_0}^{\ell+5/2}$,
$(\phi_0,a_0)\in\H_{\w_0}^{\ell+3}\times\H_{\w_0}^{\ell+2}$,
$(\phi_{10},a_{10})\in\H_{\w_0}^{\ell+2}\times\H_{\w_0}^{\ell+1} $ and
$(\phi_0^\eps,a_0^\eps)_{\eps\in(0,1]}$ bounded in
$\H_{\w_0}^{\ell+1}\times\H_{\w_0}^{\ell}$ such that 
\begin{equation*}
  r_1^\eps:=\|\phi_0^\eps- \phi_0-\eps\phi_{10}\|_{\H_{\w_0}^{\ell+1}}
  + \|a_0^\eps- a_0-\eps a_{10}\|_{\H_{\w_0}^{\ell}}=o(\eps)\text{ as
  } \eps \to 0.
\end{equation*}
Then, for $M=M(\ell+2)$ and $T=T(\ell+2)$  as in Theorem~\ref{th:wp},
there is 
an $\eps$-independent  $C>0$ such that for all $\eps\in (0,1]$, 
\begin{equation}\label{eq:est2}
\vvvert\phi^\eps-\phi-\eps \phi_1\vvvert_{\ell+1,T}+\vvvert
a^\eps-a-\eps a_1\vvvert_{\ell,T}\le C\(r_1^\eps +\eps^2\),
\end{equation}
where $(\phi^\eps,a^\eps)$ denotes the solution to
\eqref{eq:syst-grenier}, $(\phi,a)$
is the solution to \eqref{eq:syst-limit}, and $(\phi_1,a_1)$ is the solution to
\eqref{eq:syst-grenier-linearized}. In particular,
\begin{equation*}
  \left\|u^\eps - a
    e^{i\phi_1}e^{i\phi/\eps}\right\|_{L^\infty([0,T];L^2\cap
    L^\infty(\E^d))}=\O\(\frac{r_1^\eps}{\eps}+\eps\)\Tend \eps 0 0 . 
\end{equation*}
\end{theorem}
\subsection*{Outline} In Section~\ref{sec:cauchy}, we prove Theorem~\ref{th:wp}, by starting
with a generalization of key estimates established in \cite{GV01} to
the periodic setting $\E^d=\T^d$. Theorem~\ref{th:conv-obs} is
proved in Section~\ref{sec:first}, and the proof of
Theorem~\ref{th:conv-wf} is sketched in the final
Section~\ref{sec:cv}.

\section{Well-posedness}\label{sec:cauchy}
\subsection{A key bilinear estimate}
The following proposition is proved in \cite{GV01} in the case
$\E^d=\R^d$ in the context of long range scattering. We have used it
in \cite{CaGa17,CaGa17b} in the context of semi-classical analysis. We
extend it here to the case $\E^d=\T^d$.

\begin{proposition}\label{lem:tame} Let $\ell\ge 0$
and $s>d/2$. Then, 
for every $\psi_1,\psi_2\in \H_\w^{\max(\ell,s)}$,
\begin{equation}\label{eq:tame}
    \|\psi_1 \psi_2\|_{\H_\w^\ell}\le C_{\ell,s}\(
    \|\psi_1\|_{\H_\w^\ell}\|\psi_2\|_{\H_\w^s} +
    \|\psi_1\|_{\H_\w^s}\|\psi_2\|_{\H_\w^\ell}\), 
  \end{equation}
  where 
  $$C_{\ell,s}=\left\{\begin{array}{ll}
   \frac{2^\ell}{(2\pi)^{d/2}}\left\|\frac{1}{\<\cdot\>^s}\right\|_{L^2(\R^d)} &{\rm if}\quad \E^d=\R^d,\\
  \frac{2^\ell}{(2\pi)^{d/2}}\left\|\frac{1}{\<\cdot\>^s}\right\|_{\ell^2(\Z^d)} &{\rm if}\quad \E^d=\T^d.
  \end{array}\right.$$
\end{proposition}

We detail the proof only in the case $\E^d=\T^d$. The proof is
analogous in the case $\E^d=\R^d$, and can be found in \cite{GV01}
(with different notations, though). Proposition~\ref{lem:tame} in the
case $\E^d=\T^d$ stems from the following sequence of lemmas. We skip
the proofs of the most classical ones. 

\begin{lemma} \label{ineg-tri} For all $m,n\in\R^d$, \quad $\<m+n\>\le \<m\>+\<n\>.$
\end{lemma}

\begin{lemma} \label{convol} If $\psi_1,\psi_2\in L^2(\T^d)$, for all $m\in\Z^d$,
\begin{equation*}
\widehat{\psi_1\psi_2}(m)=\frac{1}{(2\pi)^{d/2}}\sum_{k\in\Z^d}\widehat{\psi_1}(k)\widehat{\psi_2}(m-k)=: \frac{1}{(2\pi)^{d/2}}\widehat{\psi_1}*\widehat{\psi_2}(m).
\end{equation*}
\end{lemma}

\begin{lemma}
\label{bi}
For $\ell,w\ge 0$, if $\psi_1,\psi_2\in \H_w^\ell$, then
\begin{align*}
\|\psi_1\psi_2\|_{\H_w^\ell}&\le \frac{2^\ell}{(2\pi)^{d/2}}\left[\left\|\(\<\cdot\>^\ell e^{w\<\cdot\>}|\widehat{\psi_1}|\)*\(e^{w\<\cdot\>}|\widehat{\psi_2}|\)\right\|_{\ell^2(\Z^d)}\right.\\
&\qquad\qquad\qquad+\left.\left\|\( e^{w\<\cdot\>}|\widehat{\psi_1}|\)*\(\<\cdot\>^\ell e^{w\<\cdot\>}|\widehat{\psi_2}|\)\right\|_{\ell^2(\Z^d)}\right]
\end{align*}
\end{lemma}
\begin{proof}
From Lemma \ref{convol},
\begin{align*}
\|\psi_1\psi_2\|_{\H_w^\ell}^2 & \le \frac{1}{(2\pi)^{d}}\sum_{m\in\Z^d}\(\sum_{n\in\Z^d}\<m\>^\ell e^{w\<m\>}|\widehat{\psi_1}(n)||\widehat{\psi_2}(m-n)|\)^2.
\end{align*}
From Lemma \ref{ineg-tri} and because for any $m,n\in\Z^d$, we have either $\<n\>\ge \<m\>/2$ or $\<m-n\>\ge \<m\>/2$, we deduce
\begin{align*}
\|\psi_1\psi_2\|_{\H_w^\ell}^2 & \le \frac{1}{(2\pi)^{d}}\sum_{m\in\Z^d}\(\sum_{n\in\Z^d,\<n\>\ge \<m\>/2}\<m\>^\ell e^{w\<n\>}|\widehat{\psi_1}(n)|e^{w\<m-n\>}|\widehat{\psi_2}(m-n)|\right.\\
&\qquad\qquad\qquad\left.+\sum_{n\in\Z^d,\<m-n\>\ge \<m\>/2}\<m\>^\ell e^{w\<n\>}|\widehat{\psi_1}(n)|e^{w\<m-n\>}|\widehat{\psi_2}(m-n)|\)^2\\
& \le \frac{2^{2\ell}}{(2\pi)^{d}}\sum_{m\in\Z^d}\(\sum_{n\in\Z^d}\<n\>^\ell e^{w\<n\>}|\widehat{\psi_1}(n)|e^{w\<m-n\>}|\widehat{\psi_2}(m-n)|\right.\\
&\qquad\qquad\qquad\left.+\sum_{n\in\Z^d}e^{w\<n\>}|\widehat{\psi_1}(n)|\<m-n\>^\ell e^{w\<m-n\>}|\widehat{\psi_2}(m-n)|\)^2.
\end{align*}
The result follows thanks to the triangle inequality in $\ell^2(\Z^d)$.
\end{proof}

\begin{lemma}
\label{young}
If $u\in \ell^2(\Z^d)$ and $v\in\ell^1(\Z^d)$, then $u*v\in \ell^2(\Z^d)$, and
\begin{align*}
\|u*v\|_{\ell^2(\Z^d)}\le \|u\|_{\ell^2(\Z^d)}\|v\|_{\ell^1(\Z^d)}
\end{align*}
\end{lemma}

\begin{proof}[Proof of Proposition \ref{lem:tame}]
Let us estimate the first term in the bracket of the right hand side of the inequality in Lemma \ref{bi}. The other term is treated similarly. According to Lemma \ref{young}, we have
\begin{align*}
&\left\|\(\<\cdot\>^\ell e^{w\<\cdot\>}|\widehat{\psi_1}|\)*\(e^{w\<\cdot\>}|\widehat{\psi_2}|\)\right\|_{\ell^2(\Z^d)}\\
&\qquad\le \left\|\<\cdot\>^\ell e^{w\<\cdot\>}|\widehat{\psi_1}|\right\|_{\ell^2(\Z^d)}\left\|e^{w\<\cdot\>}|\widehat{\psi_2}|\right\|_{\ell^1(\Z^d)}\\
&\qquad\le \left\|\<\cdot\>^\ell e^{w\<\cdot\>}|\widehat{\psi_1}|\right\|_{\ell^2(\Z^d)}\left\|\frac{1}{\<\cdot\>^s}\right\|_{\ell^2(\Z^d)}\left\|\<\cdot\>^s e^{w\<\cdot\>}|\widehat{\psi_2}|\right\|_{\ell^2(\Z^d)}\\
&\qquad\le\left\|\frac{1}{\<\cdot\>^s}\right\|_{\ell^2(\Z^d)}\|\psi_1\|_{\H_w^\ell}\|\psi_2\|_{\H_w^s},
\end{align*}
where we have also used the Cauchy-Schwarz inequality.
\end{proof}


\subsection{The iterative scheme}
In this section,  $\eps\in [0,1]$ is fixed. To lighten the notations,
we consider the case $J=1$ (only one Fourier multiplier), and leave
out the corresponding index: the proof shows that considering
finitely many such terms is straightforward. Solutions to (\ref{eq:syst-grenier}) are
constructed  as limits of the solutions of the iterative scheme  
\begin{equation}\label{eq:iteration}
  \left\{
    \begin{aligned}
      &\d_t \phi^\eps_{j+1} + \frac{1}{2}\<\nabla\phi^\eps_j,
      H\nabla\phi^\eps_{j+1}\>+g(|a_j^\eps|^2)\<\beta,\nabla
      \phi^\eps_{j+1}\>=-K\ast |a_j^\eps|^{2\sigma}-V,\\
      &\qquad \qquad\phi^\eps_{j+1\mid t=0}=\phi_0^\eps,\\
& \d_t a^\eps_{j+1} + \<\nabla\phi^\eps_j, H\nabla a^\eps_{j+1}\>
+\frac{1}{2}(D^2 \phi^\eps_j) a^\eps_{j+1}\\
&\qquad+\<\beta,\nabla\(g(|a_j^\eps|^2)\)\>a^\eps_{j+1}+h(|a_j^\eps|^2)\overline{a_j^\eps}\<\beta,\nabla a_j^\eps\> a^\eps_{j+1}= 
\frac{i\eps}{2}D^2 a^\eps_{j+1},\\ 
& \qquad \qquad a^\eps_{j+1\mid t=0}=a_0^\eps,
\end{aligned}
\right.
\end{equation}
where $h(s)=g(s)/s$.
 The scheme is initialized with the time-independent pair
$(\phi_0^\eps,a_0^\eps)\in
\H_{\w_0}^{\ell+1}\times\H_{\w_0}^{\ell}\subset X_{\w,T}^\ell$ for any
$T>0$. 
\smallbreak

The scheme is well-defined: if $\ell>(d+1)/2$, for a given $(\phi_j^\eps,a_j^\eps)\in X_{\w,T}^\ell$, (\ref{eq:iteration}) defines $(\phi_{j+1}^\eps,a_{j+1}^\eps)$.
Indeed, in the first equation, $\phi_{j+1}^\eps$
solves a linear transport equation with smooth coefficients, which
guarantees the existence of a solution $\phi_{j+1}^\eps\in
L^\infty_TL^2$ (see e.g \cite[Section 3]{BCD11},  or \cite[Section
  II.C, Proposition~1.2]{AlGe07}). The same argument provides a solution
$a_{j+1}^\eps\in L^\infty_TL^2$ to the second equation in the case
$\eps=0$.  On the other hand, if $\eps>0$, the  
second equation is equivalent through the relation $v_{j+1}^\eps = a_{j+1}^\eps
e^{i\phi_j^\eps/\eps}$ to the equation
\begin{align}\label{eq:linear-schrodinger}
& i\eps\d_t v_{j+1}^\eps +\frac{\eps^2}{2}D^2 v_{j+1}^\eps +W(t,x)v_{j+1}^\eps=0,
\end{align}
with initial condition
$$\quad v_{j+1\mid t=0}^\eps= v_0^\eps=a_0^\eps e^{i\phi_0^\eps/\eps},$$
where 
$$W=\d_t \phi_j^\eps +\frac{1}{2}\<\nabla
  \phi_j^\eps, H\nabla\phi_j^\eps\>+i\eps\<\beta,\nabla
  \(g\(|a_j^\eps|^2\)\)\>
  +i\eps h(|a^\eps_j|^2)\overline{a_j^\eps}\<\beta,\nabla
  a_j^\eps\>.$$ 
This is a linear Schr\"odinger like equation, with a second order
operator $D^2$ which is not necessarily elliptic, and with a 
smooth and bounded external time-dependent potential $W(t,x)$. Note that
this external potential is complex-valued, so the existence of a
solution for \eqref{eq:linear-schrodinger} is not quite standard.
On the other hand, a fixed
point argument applied to the map 
\[
  \Psi :\left\{
  \begin{array}{rcl}C([0,T^\eps],L^2)&\longrightarrow &C([0,T^\eps],L^2)\\u &\mapsto &e^{i\eps t D^2/2}v_0^\eps+\frac{i}{\eps}\int_0^t e^{i\eps (t-s) D^2/2}\left[W(s)u(s)\right]ds,
  \end{array}
\right.
\]
where $0<T^\eps\le T$, provides the existence of a $C([0,T^\eps],L^2)$
solution to (\ref{eq:linear-schrodinger}), for some $T^\eps>0$. We
actually have $v_{j+1}^\eps\in C([0,T],L^2)$, since $W\in
L^\infty([0,T]\times\E^d)$. 
\smallbreak

The following lemma gives the estimates that will ensure that
$(\phi_{j+1}^\eps,a_{j+1}^\eps)\in X_{\w,T}^\ell$ provided
$(\phi_j^\eps,a_j^\eps)\in X_{\w,T}^\ell$. It is almost identical to
Lemma~2.2 in \cite{CaGa17b}. 

\begin{lemma}\label{lem:evol-norm}
 Let  $\ell>(d+1)/2$ and $T>0$. Let $(\phi,a)\in X_{\w,T}^\ell$,
 $\tilde{a}\in Y_{\w,T}^{\ell+1}$ and $(\FF,\GG)\in
 L^2([0,T],\H_\w^{\ell+1/2}\times \H_\w^{\ell-1/2})$ such that 
\begin{align}
&\d_t\phi=\FF,& \phi(0)\in\H_{\w_0}^{\ell+1},\label{eq:phi}\\
&\d_t a=\GG+i\theta_1D^2 a+i\theta_2D^2\tilde{a},& a(0)\in\H_{\w_0}^{\ell},\label{eq:a}
\end{align}    
where $\theta_1,\theta_2\in\R$. Then
\begin{align}
&\vvvert\phi\vvvert_{\ell+1,T}^2  \le
  \|\phi(0)\|_{\H_{\w_0}^{\ell+1}}^2+\frac{1}{M}\vvvert\phi\vvvert_{\ell+1,T}\sqrt{2M}\|\FF\|_{L^2_T\H_{\w}^{\ell+1/2}},\label{eq:phiY}\\ 
&\vvvert a\vvvert_{\ell,T}^2  \le
  \|a(0)\|_{\H_{\w_0}^{\ell}}^2+\frac{1}{M}\vvvert
  a\vvvert_{\ell,T}\sqrt{2M}\|\GG\|_{L^2_T\H_{\w}^{\ell-1/2}}+
\frac{|\theta_2|}{2M}\vvvert 
  a\vvvert_{\ell,T}\vvvert \tilde{a}\vvvert_{\ell+1,T}.\label{eq:aY} 
\end{align}
Moreover, there exists $C>0$ (that depends only on $\ell$, not on $w$) such that
\begin{itemize}
\item If $\FF=\d_j\psi_1 \d_k\psi_2$ with $\psi_1,\psi_2\in
  Y_{\ell+1,T}$, then 
\begin{equation}\label{eq:mu1} 
\sqrt{2M}\|\FF\|_{L^2_T\H_{\w}^{\ell+1/2}}\le C\vvvert
\psi_1\vvvert_{\ell+1,T}\vvvert \psi_2\vvvert_{\ell+1,T} .
\end{equation}
\item If
  $\FF=\left(\overset{2n}{\underset{j=1}{\prod}}b_j\right)\d_k\psi$
  with $n\ge 1$, $\psi\in Y_{\ell+1,T}$ and $b_j\in Y_{\ell,T}$ for all $j$,
  then 
\begin{equation}\label{eq:mu2}
\sqrt{2M}\|\FF\|_{L^2_T\H_{\w}^{\ell+1/2}}\le C\left(\overset{2n}{\underset{j=1}{\prod}} \vvvert b_j\vvvert_{\ell,T}\right) \vvvert\psi\vvvert_{\ell+1,T}.
\end{equation}
\item If $\FF=K*\left(\overset{2n}{\underset{j=1}{\prod}} b_j\right)$ with $n\ge 1$, $b_j\in Y_{\ell,T}$ for all $j$ and $\widehat{K}$ uniformly bounded, then
\begin{equation}\label{eq:mu3}
\sqrt{2M}\|\FF\|_{L^2_T\H_{\w}^{\ell+1/2}}\le C\left(\overset{2n}{\underset{j=1}{\prod}} \vvvert b_j\vvvert_{\ell,T}\right).
\end{equation}
\item If $\GG=\d_j\psi \d_k b$ with $\psi\in Y_{\ell+1,T}$ and $b\in Y_{\ell,T}$, then
\begin{equation}\label{eq:nu1}
\sqrt{2M}\|\GG\|_{L^2_T\H_{\w}^{\ell-1/2}}\le C\vvvert \psi\vvvert_{\ell+1,T}\vvvert b\vvvert_{\ell,T}.
\end{equation}
\item If $\GG=b D^2\psi $ with $\psi\in Y_{\ell+1,T}$ and $b\in Y_{\ell,T}$, then
\begin{equation}\label{eq:nu2}
\sqrt{2M}\|\GG\|_{L^2_T\H_{\w}^{\ell-1/2}}\le C\vvvert \psi\vvvert_{\ell+1,T}\vvvert b\vvvert_{\ell,T}.
\end{equation}
\item If $\GG=\left(\overset{2n}{\underset{j=1}{\prod}}b_j\right)\d_k b$ with $n\ge 1$, $b, b_j\in Y_{\ell,T}$ for all $j$, then
\begin{equation}\label{eq:nu3}
\sqrt{2M}\|\GG\|_{L^2_T\H_{\w}^{\ell-1/2}}\le C\left(\overset{2n}{\underset{j=1}{\prod}} \vvvert b_j\vvvert_{\ell,T}\right) \vvvert b\vvvert_{\ell,T}.
\end{equation}

\end{itemize}
\end{lemma}

\begin{proof}
The proof of (\ref{eq:phiY}) and (\ref{eq:aY}) is identical to the one
given in \cite{CaGa17b}. The new constraint $\ell>(d+1)/2$ plays no
role here. Inequalities similar to (\ref{eq:mu1})-(\ref{eq:nu3}) were
proved in \cite{CaGa17b}. The only differences with \cite{CaGa17b} are
the presence of the kernel $K$ in (\ref{eq:mu3}) and the constraint on
$s$ in Proposition \ref{eq:tame}, which is $s>d/2$ whereas it was
$s>1/2$ in \cite{CaGa17b} (where we had $d=1$). It is actually sufficient to assume $\ell>d/2$  for the proof of (\ref{eq:mu1})-(\ref{eq:mu3}), as we use several times (\ref{eq:tame}) with $m=\ell+1/2$ and $s=\ell>d/2$, or with $m=s=\ell$.
For instance (\ref{eq:mu3}) follows from
\begin{align*}
\left\|K*\left(\overset{2n}{\underset{j=1}{\prod}}
  b_j\right)\right\|_{\H_{\w}^{\ell+1/2}}&\le
                                           \|\widehat{K}\|_{L^\infty}
                                           \left\|\overset{2n}{\underset{j=1}{\prod}} 
                                           b_j\right\|_{\H_{\w}^{\ell+1/2}}\\  
&\le C\sum_{j=1}^{2n}\prod_{k=1}^{j-1}\|b_k\|_{\H_{\w}^{\ell}}\|b_j\|_{\H_{\w}^{\ell+1/2}}\prod_{k=j+1}^{2n}\|b_k\|_{\H_{\w}^{\ell}}.
\end{align*}
In order to prove (\ref{eq:nu1})-(\ref{eq:nu3}), we use (\ref{eq:tame}) with $m=s=\ell-1/2>d/2$ which is possible thanks to the assumption $\ell>(d+1)/2$.
Actually, even (\ref{eq:nu1})-(\ref{eq:nu3}) can be proved under the condition $\ell>d/2$, thanks to a
refined version of Lemma \ref{lem:tame} (see \cite{CaGa17}). However,
since it is not useful in the sequel to sharpen this assumption, we
choose to make the stronger assumption $\ell>(d+1)/2$ for the sake of conciseness. 
\end{proof}

\noindent {\bf First step: boundedness of the sequence.}
In view of the equation satisfied by $\phi_{j+1}^\eps$ in
(\ref{eq:iteration}), Lemma~\ref{lem:evol-norm} yields
\begin{align*}
  \vvvert  \phi_{j+1}^\eps\vvvert_{\ell+1,T}^2 &\le \|
  \phi_0^\eps\|_{\H_{\w_0}^{\ell+1}}^2 +\frac{C}{M} 
\vvvert \phi_{j+1}^\eps\vvvert_{\ell+1,T}^2 
\vvvert \phi_{j}^\eps\vvvert_{\ell+1,T}
+\frac{C}{M} 
\vvvert \phi_{j+1}^\eps\vvvert_{\ell+1,T}
\vvvert a_{j}^\eps\vvvert_{\ell,T}^{2\si}\\
&\quad +\frac{C}{M} \vvvert \phi_{j+1}^\eps\vvvert_{\ell+1,T}^2\vvvert a_{j}^\eps\vvvert_{\ell,T}^{2\gamma}+\frac{C}{M} \vvvert \phi_{j+1}^\eps\vvvert_{\ell+1,T}\sqrt{2M}\| V\|_{L^2_T\H_w^{\ell+1/2}}.
\end{align*}
As for $a_{j+1}^\eps$, we obtain in a similar way
\begin{align*}
  \vvvert  a_{j+1}^\eps\vvvert_{\ell,T}^2 &\le \|
 a_0^\eps\|_{\H_{\w_0}^\ell}^2 +\frac{C}{M} 
\vvvert a_{j+1}^\eps\vvvert_{\ell,T}^2 
\vvvert \phi_{j}^\eps\vvvert_{\ell+1,T}+\frac{C}{M} 
\vvvert a_{j+1}^\eps\vvvert_{\ell,T}^2 
\vvvert a_{j}^\eps\vvvert_{\ell,T}^{2\gamma}.
\end{align*}
Up to the term with $V$ in the first one, the last two estimates are exactly the ones we had in \cite{CaGa17b}. The proof of the boundedness of the sequence $(\phi_j^\eps,a_j^\eps)$ in $X_{w,T}^\ell$ is quite similar to what was done in \cite{CaGa17b}. Indeed, under the assumption 
\begin{equation}\label{eq:rec}
\frac{C}{M}\vvvert\phi_j^\eps\vvvert_{\ell+1,T}\le\frac{1}{4},\qquad \frac{C}{M}\vvvert a_j^\eps\vvvert_{\ell,T}^{2\gamma}\le\frac{1}{4},
\end{equation}
we have
\begin{align}\label{eq-hyp-rec-j+1}
\frac{1}{4}\vvvert\phi_{j+1}^\eps\vvvert_{\ell+1,T}^2
  &\le\|\phi_0^\eps\|_{\H_{\w_0}^{\ell+1}}^2+\frac{2C^2}{M^2}\vvvert
    a_j^\eps\vvvert_{\ell,T}^{4\sigma}+\frac{4C^2}{M}\|V\|_{L^2_T\H_w^{\ell+1/2}}^2 
\end{align}
and
\begin{align}\label{eq-hyp-rec-j+1-bis}
\frac{1}{2}\vvvert a_{j+1}^\eps\vvvert_{\ell,T}^2 & \le\|a_0^\eps\|_{\H_{\w_0}^{\ell}}^2.
\end{align}
Note that the monotonicity property \eqref{eq:monotone} implies
\begin{equation*}
  \|V\|_{L^2_T\H_w^{\ell+1/2}} \le \|V\|_{L^2_{T_0}\H_{w_0}^{\ell+1/2}}.
\end{equation*}
We next show by induction that, provided $M$ is sufficiently
large, we can construct a sequence $(\phi_j^\eps, a_j^\eps)_{j\in\N}$
such that for every $j\in\N$,  
\begin{align}\label{eq:hyp-rec}
\vvvert\phi_{j}^\eps\vvvert_{\ell+1,T}^2
  &\le4\|\phi_0^\eps\|_{\H_{\w_0}^{\ell+1}}^2+\frac{8C^2}{M^2}\(2\|a_0^\eps\|_{\H_{\w_0}^{\ell}}^2\)^{2\sigma}+\frac{16C^2}{M}\|V\|_{L^2_{T_0}\H_{w_0}^{\ell+1/2}}^2, 
\end{align}
\begin{align}\label{eq:hyp-rec-bis}
\vvvert a_{j}^\eps\vvvert_{\ell,T}^2 &\le 2\|a_0^\eps\|_{\H_{\w_0}^{\ell}}^2.
\end{align}

For that purpose, we choose $M$ sufficiently large such that
(\ref{eq:rec}) holds for $j=0$ and such that  
\begin{align}\label{eq:cond-M}
4\|\phi_0^\eps\|_{\H_{\w_0}^{\ell+1}}^2+\frac{8C^2}{M^2}\(2\|a_0^\eps\|_{\H_{\w_0}^{\ell}}^2\)^{2\sigma}+\frac{16C^2}{M}\|V\|_{L^2_{T_0}\H_{w_0}^{\ell+1/2}}^2\le \frac{M^2}{16C^2},
\end{align}
and
\begin{align}\label{eq:cond-M-bis}
(2\|a_0^\eps\|_{\H_{\w_0}^{\ell}}^2)^\gamma\le \frac{M}{4C}.
\end{align}
Then, (\ref{eq:hyp-rec})-(\ref{eq:hyp-rec-bis}) hold for $j=0$, since with
$(\phi_0^\eps,a_0^\eps)(t,x)=(\phi_0^\eps,a_0^\eps)(x)$ independent of
time,  it is
easy to check that $\vvvert
\phi_0^\eps\vvvert_{\ell+1,T}= \|\phi_0^\eps\|_{\H_{\w_0}^{\ell+1}}$
and $\vvvert
a_0^\eps\vvvert_{\ell,T}=\|a_0^\eps\|_{\H_{\w_0}^{\ell}}$. Let $j\ge
0$ and assume that (\ref{eq:hyp-rec})-(\ref{eq:hyp-rec-bis}) hold. Then (\ref{eq:hyp-rec})-(\ref{eq:hyp-rec-bis})
and (\ref{eq:cond-M})-(\ref{eq:cond-M-bis}) ensure that the condition (\ref{eq:rec}) is
satisfied, and therefore (\ref{eq-hyp-rec-j+1})-(\ref{eq-hyp-rec-j+1-bis}) hold, from which we
infer easily that (\ref{eq:hyp-rec})-(\ref{eq:hyp-rec-bis}) are true for $j$ replaced by
$j+1$ (for the estimate on the norm  of $V$, we use here the fact that $T\le T_0$ and $w\le w_0$).
\smallbreak

\noindent {\bf Second step: convergence.}
For $j\ge 1$, we set
$\delta\phi_j^\eps=\phi_j^\eps-\phi_{j-1}^\eps$, and 
$\delta a_j^\eps=a_j^\eps-a_{j-1}^\eps$. Then, for every $j\ge
1$, we have 
\begin{align*}
  &\d_t \delta\phi_{j+1}^\eps+ \frac{1}{2}\(\<\nabla \phi^\eps_j,
     H \nabla \delta \phi^\eps_{j+1}  \>+ \<\nabla \delta \phi^\eps_j,
      H\nabla\phi^\eps_{j}\>\)+g\(|a_j^\eps|^2\)\<\beta,\nabla\delta\phi_{j+1}^\eps\>\\
      &\qquad+\(g\(|a_j^\eps|^2\)-g\(|a_{j-1}^\eps|^2\)\)\<\beta,\nabla\phi^\eps_{j} \>+K*\(|a_j^\eps|^{2\sigma}-|a_{j-1}^\eps|^{2\sigma}\)=0.
\end{align*}
and
\begin{align*}
&  \d_t \delta a_{j+1}^\eps + \<\nabla \phi_j^\eps, H\nabla \delta
  a_{j+1}^\eps \>+ \<\nabla \delta \phi_j^\eps , H\nabla a_j^\eps \>+
  \frac{1}{2}\delta a_{j+1}^\eps D^2 \phi_j^\eps +
  \frac{1}{2}a_j^\eps D^2 \delta \phi_j^\eps \\
&\quad  +\<\beta,\nabla \(g\(|a_j^\eps|^2\)\)\>\delta a^\eps_{j+1}+
  \<\beta,\nabla\(g\(|a_j^\eps|^2\)-g\(|a_{j-1}^\eps|^2\)\)\>a_j^\eps\\
 &\quad  +h\(|a^\eps_j|^2\)\<\beta,\nabla a_j^\eps\>\overline{a_j^\eps} \delta a^\eps_{j+1}
 + h\(|a^\eps_j|^2\)\<\beta,\nabla a_j^\eps\>\overline{\delta a_j^\eps}  a^\eps_{j}\\
&\quad + h\(|a^\eps_j|^2\)\<\beta,\nabla \delta a_j^\eps\>\overline{ a_{j-1}^\eps} a^\eps_{j}
 + \(h\(|a^\eps_j|^2\)-h\(|a^\eps_{j-1}|^2\)\)\<\beta,\nabla  a_{j-1}^\eps\>\overline{ a_{j-1}^\eps} a^\eps_{j}\\
&\quad  = i\frac{\eps}{2}D^2
  \delta a_{j+1}^\eps, 
\end{align*}
Lemma \ref{lem:evol-norm} and the boundedness of $(\phi_j^\eps,a_j^\eps)$ in $X_{w,T}^\ell$ imply like in \cite{CaGa17b} that for $M$ large enough,
$$\max\(\vvvert  \delta\phi_{j+1}^\eps\vvvert_{\ell+1,T}^2,\vvvert  \delta a_{j+1}^\eps\vvvert_{\ell,T}^2\)\le \frac{K}{M}\( \vvvert  \delta\phi_{j}^\eps\vvvert_{\ell+1,T}^2 +
 \vvvert  \delta a_{j}^\eps\vvvert_{\ell,T}^2 \)$$
for some $K>0$ which does not depend on $\eps$ provided
$(\phi_0^\eps, a_0^\eps)_{\eps\in [0,1]}$ is uniformly bounded in
$\H_{\w_0}^{\ell+1}\times \H_{\w_0}^{\ell}$. We conclude as in \cite{CaGa17b} that provided $\ell>(d+1)/2$, possibly
increasing $M$, $(\phi_j^\eps,a_j^\eps)$ converges geometrically 
in $X_{\w,T}^\ell$
as $j\to \infty$. Uniqueness of the solution $(\phi^\eps, a^\eps)$ to
(\ref{eq:syst-grenier}) follows from the same kind of estimates as the
ones which prove the convergence. 

\section{First order approximation}
\label{sec:first}

As in the previous section, we assume $J=1$ for the sake of
conciseness. 
\begin{proof}[Proof of Theorem~\ref{th:conv-obs}]
Next, assume that
$(\phi_0,a_0)\in\H_{\w_0}^{\ell+2}\times\H_{\w_0}^{\ell+1} $. Then, in
view of Theorem~\ref{th:wp},
the solution $(\phi,a)$ to  \eqref{eq:syst-limit} belongs to
$X_{\w,T}^{\ell+1}$. Given $\eps>0$, if $(\phi_0^\eps,a_0^\eps)\in
\H_{\w_0}^{\ell+1}\times\H_{\w_0}^{\ell}$, we denote by
$(\phi^\eps,a^\eps)$ the solution to (\ref{eq:syst-grenier}). We also denote
$(\delta\phi^\eps, \delta a^\eps)=(\phi^\eps-\phi, a^\eps-a)$.  Then,
in the same fashion as above, we have 
\begin{align*}
  &\d_t \delta\phi^\eps+ \frac{1}{2}\(\<\nabla \delta\phi^\eps,
      H\nabla \phi^\eps\>  + \<\nabla\phi, H\nabla \delta
    \phi^\eps\>\)+g\(|a^\eps|^2\)\<\beta,\nabla\delta\phi^\eps\>\\ 
      &\qquad+\(g\(|a^\eps|^2\)-g\(|a|^2\)\)\<\beta,\nabla\phi\> +K*\(|a^\eps|^{2\sigma}-|a|^{2\sigma}\)=0
\end{align*}

and
\begin{align*}
&  \d_t \delta a^\eps + \<\nabla \delta\phi^\eps, H\nabla
  a^\eps\> + \<\nabla \phi, H\nabla \delta a^\eps\> +
  \frac{1}{2}\delta a^\eps D^2 \phi^\eps +
  \frac{1}{2}a D^2 \delta \phi^\eps \\
&\quad  +\<\beta,\nabla \(g\(|a^\eps|^2\)\)\>\delta a^\eps+
  \<\beta,\nabla\(g\(|a^\eps|^2\)-g\(|a|^2\)\)\>a\\
 &\quad  + h\(|a^\eps|^2\)\<\beta,\nabla a^\eps\>\overline{a^\eps} \delta a^\eps
 + h\(|a^\eps|^2\)\<\beta,\nabla a^\eps\>\overline{\delta a^\eps}  a\\
&\quad + h\(|a^\eps|^2\)\<\beta,\nabla \delta a^\eps\>|a|^2
 +\(h\(|a^\eps|^2\)-h\(|a|^2\)\)\<\beta,\nabla  a\>|a|^2  = i\frac{\eps}{2}D^2
  \delta a^\eps+i\frac{\eps}{2}D^2 a.
\end{align*}

Like in \cite{CaGa17b}, for some new constant $k$, Lemma \ref{lem:evol-norm} and Theorem \ref{th:wp} imply, for $M$ large enough, 
\begin{equation*}
    \vvvert  \delta\phi^\eps\vvvert_{\ell+1,T}^2  \le 
k\|\phi_0^\eps-\phi_0\|_{\H_{\w_0}^{\ell+1}}^2+\frac{k}{M}
 \vvvert  \delta a^\eps\vvvert_{\ell,T}^2,
\end{equation*}
and
\begin{equation*}
    \vvvert  \delta a^\eps\vvvert_{\ell,T}^2  \le 
k\|a_0^\eps-a_0\|_{\H_{\w_0}^\ell}^2+\frac{k}{M} 
\vvvert  \delta\phi^\eps\vvvert_{\ell+1,T}^2+\frac{k}{M}\eps \vvvert\delta a^\eps\vvvert_{\ell,T}\vvvert a\vvvert_{\ell+1,T}.
\end{equation*}
Possibly increasing the value of $M$ and adding the last two inequalities, we deduce
\begin{equation*}
    \vvvert  \delta\phi^\eps\vvvert_{\ell+1,T}^2+\vvvert  \delta a^\eps\vvvert_{\ell,T}^2\le C\|\phi_0^\eps-\phi_0\|_{\H_{\w_0}^{\ell+1}}^2+C\|a_0^\eps-a_0\|_{\H_{\w_0}^\ell}^2+C\eps^2,
\end{equation*}
hence Theorem \ref{th:conv-obs}. As for the choice of $M$, a careful examination of the previous inequalities shows that aside from the assumption $M\ge M(\ell+1)$, which enables to estimate the source term, $M$ can be chosen as in Theorem \ref{th:wp}, namely such that $M\ge M(\ell)$.
\end{proof}

\begin{proof}[Proof of Corollary~\ref{cor:quad}]
Notice that,  provided  $\w\ge 0$,
\begin{equation}\label{eq:obvious}
  \|\psi\|_{H^\ell(\E^d)}\le \|\psi\|_{\H_\w^\ell}.
\end{equation}
In particular, Sobolev embedding yields, for $\ell>(d+1)/2\ge 1$,
\begin{equation*}
  \|\psi\|_{L^\infty(\E^d)}\le C \|\psi\|_{\H_\w^\ell},
\end{equation*}
where $C$ is independent of $\w \ge 0$. 
With these remarks in mind, the $L^1$ estimates of
Corollary~\ref{cor:quad} follow from Theorem~\ref{th:conv-obs} and 
Cauchy-Schwarz inequality, since
\begin{equation*}
  \left\| |u^\eps|^2-|a|^2\right\|_{L^\infty_TL^1} =  \left\|
    |a^\eps|^2-|a|^2\right\|_{L^\infty_TL^1}  \le
  \|a^\eps+a\|_{L^\infty_T L^2} \|\delta a^\eps\|_{L^\infty_TL^2},
\end{equation*}
and
\begin{align*}
  \|\IM \(\eps \bar u^\eps \d u^\eps\)
  -|a|^2\d\phi\|_{L^\infty_TL^1} &\le \eps\|\IM \bar a^\eps \d
  a^\eps\|_{L^\infty_TL^1}  + \||a^\eps|^2\d\phi^\eps-|a|^2\d
  \phi\|_{L^\infty_TL^1}  \\
&\le \eps \|a^\eps\|^2_{L^\infty_TH^1}
  +\|a^\eps+a\|_{L^\infty_T L^2} \|\delta a^\eps\|_{L^\infty_TL^2}
  \|\d \phi\|_{L^\infty_TL^\infty} \\
&\quad+ \|a^\eps\|_{L^\infty_TL^\infty}
\|a^\eps\|_{L^\infty_TL^2}\|\delta \phi^\eps\|_{L^\infty_TH^1}.
\end{align*}
The $L^\infty$ estimates in space follow by replacing $L^1$ and $L^2$
by $L^\infty$ in the above inequalities, and using Sobolev embedding
again. 
\end{proof}

\section{Convergence of the wave function}
\label{sec:cv}
Again, we assume $J=1$ for the sake of conciseness.
\begin{proof}[Proof of Theorem~\ref{th:conv-wf}]
Let $\ell>(d+1)/2$, and
$(\phi_0,a_0)\in\H_{\w_0}^{\ell+2}\times\H_{\w_0}^{\ell+1}$. Theorem~\ref{th:wp}
yields a unique solution
$(\phi,a)\in X_{\w,T}^{\ell+1}$  to \eqref{eq:syst-limit}.
\smallbreak

Let  $(\phi_{10},a_{10})\in \mathcal{H}_{\w_0}^{\ell+1}\times
\mathcal{H}_{\w_0}^{\ell}$. 
Note that
\eqref{eq:syst-grenier-linearized} is a system of linear transport
equations in the unknown $(\nabla\phi_1,a_1)$, whose coefficients are smooth functions. The general theory
of transport equations (see e.g. \cite[Section~3]{BCD11}) then shows  that
\eqref{eq:syst-grenier-linearized} has a unique solution
$(\phi_1,a_1)\in C([0,T],L^2\times L^2)$. We already know by this
argument that the solution is actually more regular (in terms of
Sobolev regularity), but we shall
directly use a priori estimates in $\H^\ell_\w$ spaces. Indeed,
Lemma \ref{lem:evol-norm} implies that $(\phi_1,a_1)\in X_{\w,T}^\ell$ with, exactly as in \cite{CaGa17b},
\begin{align*}
&\vvvert\phi_1\vvvert_{\ell+1,T}^2\le \|\phi_{10}\|_{\mathcal{H}_{\w_0}^{\ell+1}}^2
+\frac{C}{M}\vvvert\phi_1\vvvert_{\ell+1,T}^2\vvvert\phi\vvvert_{\ell+1,T}+\frac{C}{M}\vvvert\phi_1\vvvert_{\ell+1,T}^2\vvvert a\vvvert_{\ell,T}^{2\gamma}\\
&\qquad+\frac{C}{M}\vvvert\phi_1\vvvert_{\ell+1,T}\vvvert\phi\vvvert_{\ell+1,T}\vvvert a\vvvert_{\ell,T}^{2\gamma-1}\vvvert a_1\vvvert_{\ell,T}+\frac{C}{M}\vvvert\phi_1\vvvert_{\ell+1,T}\vvvert a\vvvert_{\ell,T}^{2\sigma-1}\vvvert 
a_1\vvvert_{\ell,T},
\end{align*}
along with
\begin{align*}
\vvvert a_1\vvvert_{\ell,T}^2&\le \|a_{10}\|_{\mathcal{H}_{\w_0}^{\ell}}^2+\frac{C}{M}\vvvert a_1\vvvert_{\ell,T}\vvvert a\vvvert_{\ell,T}\vvvert\phi_1\vvvert_{\ell+1,T}+\frac{C}{M}\vvvert a_1\vvvert_{\ell,T}^2\vvvert\phi\vvvert_{\ell+1,T}\\
&\qquad+\frac{C}{M}\vvvert a_1\vvvert_{\ell,T}^2\vvvert a\vvvert_{\ell,T}^{2\gamma}+\frac{C}{M}\vvvert a_1\vvvert_{\ell,T}\vvvert a\vvvert_{\ell+1,T},
\end{align*}
for some $C>0$. 

Let $\ell>(d+1)/2$. For $(\phi_0,a_0)\in \H_{\w_0}^{\ell+3}\times
\H_{\w_0}^{\ell+2}$,  $(\phi_{10},a_{10})\in
\H_{\w_0}^{\ell+2}\times \H_{\w_0}^{\ell+1}$ and $(\phi_0^\eps,a_0^\eps)\in \H_{\w_0}^{\ell+1}\times
\H_{\w_0}^{\ell}$, we consider:
\begin{itemize}
\item $(\phi,a)\in X_{\w,T}^{\ell+2}$ the solution to \eqref{eq:syst-limit}.
\item $(\phi_1,a_1)\in X_{\w,T}^{\ell+1}$ the solution to
  \eqref{eq:syst-grenier-linearized}.
\item $(\phi_{\rm app}^\eps,a_{\rm app}^\eps)=(\phi,a)+\eps(\phi_1,a_1)$.
\item $(\phi^\eps,a^\eps)\in X_{\w,T}^\ell$ the solution to
  \eqref{eq:syst-grenier}.
\end{itemize}
We assume that $\|\phi_0^\eps- \phi_0-\eps\phi_{10}\|_{\H_{\w_0}^{\ell+1}}=o(\eps)$ and $\|a_0^\eps-a_0-\eps a_{10}\|_{\H_{\w_0}^\ell}=o(\eps)$.
Set 
\begin{align*}
\delta\phi_1^\eps &=\phi^\eps-\phi_{\rm app}^\eps=\phi^\eps-\phi-\eps\phi_1=\delta\phi^\eps-\eps\phi_1,\\
\delta a_1^\eps &=a^\eps-a_{\rm app}^\eps=a^\eps-a-\eps a_1=\delta
                  a^\eps-\eps a_1.
\end{align*}
The equation satisfied by $\delta\phi_1^\eps$ writes
\begin{align*}
  &\d_t \delta\phi_1^\eps+ \<\nabla\phi, H\nabla \delta \phi_1^\eps \>+ \frac{1}{2}\<\nabla\delta \phi^\eps, H\nabla \delta \phi^\eps\>
  +g(|a|^2)\<\beta,\nabla\delta\phi_1^\eps\>\\
&  +
  \(g(|a^\eps|^2)-g(|a|^2)\)\<\beta,\nabla\delta\phi^\eps\>
  +\(g(|a^\eps|^2)-g(|a|^2)-2g'(|a|^2)\RE(\overline{a}\eps a_1)\)\<\beta,\nabla\phi\>\\
&  +K*\(|a^\eps|^{2\sigma}-|a|^{2\sigma}-2\sigma|a|^{2\sigma-2}\RE(\overline{a}\eps a_1)\)=0.
\end{align*}
Moreover, the Taylor formula yields
\begin{align}\label{taylor}
&g(|a^\eps|^2)-g(|a|^2)-2g'(|a|^2)\RE(\overline{a}\eps a_1)=\(2\RE(\overline{a}\delta a_1^\eps)+|\delta a^\eps|^2\)g'\(|a|^2\)\nonumber\\
&\qquad+\(|a^\eps|^2-|a|^2\)^2\int_0^1(1-s)g''\(|a|^2+s\(|a^\eps|^2-|a|^2\)\)ds,
\end{align}
and the same identity holds for $g$ replaced by $f(r)=r^{\sigma}$. Thus, taking into
account Theorem~\ref{th:wp}, which implies
$\vvvert\phi^\eps\vvvert_{\ell+1,T},\vvvert
a^\eps\vvvert_{\ell,T}=\O(1)$, and Theorem \ref{th:conv-obs}, which
implies $\vvvert\delta\phi^\eps\vvvert_{\ell+1,T},\vvvert \delta
a^\eps\vvvert_{\ell,T}=\O(\eps)$, it follows from Lemma
\ref{lem:evol-norm} that  
\begin{align*}
   \vvvert  \delta\phi_{1}^\eps\vvvert_{\ell+1,T}^2 
   &\le \|\phi_0^\eps-\phi_0-\eps\phi_{1,0}\|_{\H_{\w_0}^{\ell+1}}  ^2+\frac{C}{M}\vvvert\delta\phi_1^\eps\vvvert_{\ell+1,T} \left[\vvvert\delta\phi_1^\eps\vvvert_{\ell+1,T}+\eps^2+\vvvert\delta a_1^\eps\vvvert_{\ell,T}\right].
\end{align*}
We deduce, for $M$ large enough,
\begin{align}\label{eq:df1}
\vvvert  \delta\phi_{1}^\eps\vvvert_{\ell+1,T}^2 
&\le C\|\phi_0^\eps-\phi_0-\eps\phi_{10}\|_{\H_{\w_0}^{\ell+1}}  ^2+\frac{C}{M}\eps^4+\frac{C}{M}\vvvert\delta a_1^\eps\vvvert_{\ell,T}^2.
\end{align}
Similarly, $\delta a_1^\eps$ solves
\begin{align*}
&\d_t \delta a_1^\eps+\<\nabla\phi, H\nabla\delta a_1^\eps\>+\<\nabla\delta\phi_1^\eps, H\nabla a\>+\<\nabla\delta\phi^\eps,H\nabla \delta a^\eps\>\\
&\quad+\frac{1}{2}a D^2\delta\phi_1^\eps+\frac{1}{2}\delta a_1^\eps D^2\phi+\frac{1}{2}\delta a^\eps D^2\delta\phi^\eps\\
&\quad+\<\beta,\nabla\left[\left(g(|a^\eps|^2)-g(|a|^2)-2\eps g'(|a|^2)\RE(\overline{a}a_1)\right)a\right]\>\\
&\quad+\<\beta,\nabla\left[\left(g(|a^\eps|^2)-g(|a|^2)\right)\eps a_1\right]\>+\<\beta,\nabla\left[g(|a^\eps|^2)\delta a_1^\eps\right]\>
=\frac{i\eps^2}{2}D^2a_1+\frac{i\eps}{2} D^2\delta a_1^\eps.
\end{align*}
From (\ref{taylor}),  Theorems~\ref{th:wp} and \ref{th:conv-obs}, and
Lemma~\ref{lem:evol-norm}, we deduce
\begin{align}\label{eq:da1}
\vvvert\delta a_1^\eps\vvvert_{\ell,T}^2\leq C\|a_0^\eps-a_0-\eps
  a_{10}\|_{\H_{\w_0}^\ell}^2+\frac{C}{M}\eps^4+\frac{C}{M}\vvvert\delta
  \phi_1^\eps\vvvert_{\ell+1,T}^2. 
\end{align}
Adding (\ref{eq:df1}) and (\ref{eq:da1}), \eqref{eq:est2} follows. Like in the proof of Theorem \ref{th:conv-obs}, a
careful examination of the inequalities that we have used shows that
all the above estimates are valid provided that we assume $M\ge M(\ell)$, the
constant provided by Theorem~\ref{th:wp}, and also $M\ge \max(M(\ell+1),M(\ell+2))$ in order to estimate the source terms. 
\smallbreak

To
complete the proof of 
Theorem \ref{th:conv-wf}, consider the point-wise estimate
\begin{align*}
  \left|a^\eps e^{i\phi^\eps/\eps}-a e^{i\phi_1}e^{i\phi/\eps}
  \right|&\le \left|a^\eps-a \right|
+\left|a^\eps \right|
\left|e^{i\phi^\eps/\eps}-e^{i(\phi+\eps\phi_1)/\eps} \right|\\
&\le \left|a^\eps-a \right|
+\left|a^\eps \right|
\left|2\sin\(\frac{\phi^\eps-\phi-\eps\phi_1 }{2\eps}\)\right|\\
&\le \left|\delta a^\eps \right|
+\frac{1}{\eps}\left|a^\eps \right|
\left|\delta \phi_1^\eps\right|.
\end{align*}
We then conclude like in the proof of Corollary~\ref{cor:quad}, by
using Cauchy-Schwarz inequality, \eqref{eq:obvious}, and Sobolev
embedding.  
\end{proof}

\bibliographystyle{siam}
\bibliography{biblio}

\end{document}